\documentclass[12pt,letterpaper,draft,german,english]{article}
\usepackage{amsmath,amsfonts}
\usepackage{babel,djgdates}

\newcommand{\djgbb}[1]{\mathbb{#1}}
\newcommand{\f}[1][p]{\djgbb{F}_{\scriptstyle #1}}
\newcommand{\C}{\djgbb{C}}

\newcommand{\z}{\djgbb{Z}}
\newcommand{\x}{V}

\newcommand{\I}{\mathcal{I}}
\newcommand{\IT}[1]{I (#1)}
\newcommand{\V}{\mathcal{Y}}
\newcommand{\Z}{\mathcal{Z}}

\newcommand{\ka}{\kappa}
\newcommand{\rmH}{\mathrm{H}}
\newcommand{\slh}{\text{\sl h}}
\newcommand{\h}[2]{\slh^{#1} (#2)}
\newcommand{\slch}{\text{\sl ch}}
\newcommand{\ch}[1]{\slch (#1)}
\newcommand{\TY}{\mathcal{T}}
\newcommand{\coho}[3]{\rmH^{#1}(#2,#3 )}
\newcommand{\hz}[1]{\coho*{#1}{\z}}
\newcommand{\hf}[3][p]{\coho{#2}{#3}{\f[#1]}}
\newcommand{\ex}[2][p]{{#1}^{1+{#2}}_+}
\newcommand{\mult}[1]{{#1}^{\times}}

\newcommand{\extr}[1]{extraspecial $#1$-group}
\newcommand{\cf}[1]{cf.~#1}
\newcommand{\elem}[1]{elementary abelian $#1$-group}
\newcommand{\elems}[1]{elementary abelian $#1$-subgroup}
\newcommand{\maxel}[1]{maximal elementary abelian subgroup}
\newcommand{\ti}[1]{totally isotropic}
\newcommand{\tis}[1]{\ti. subspace}
\newcommand{\mti}[1]{maximal \ti.}
\newcommand{\mtis}[1]{\mti. subspace}
\newcommand{\Quillen}[1]{Quillen's Theorem}
\newcommand{\Kuenneth}[1]{K\"unneth}
\newtheorem{theorem}{Theorem}
\newtheorem{prop}[theorem]{Proposition}
\newtheorem{lemma}[theorem]{Lemma}
\newtheorem{coroll}[theorem]{Corollary}
\newenvironment{proof}{\begin{trivlist} \item[\hskip \labelsep{\bf Proof:}]}%
{\foorp \end{trivlist}}
\newenvironment{proofof}[1]{\begin{trivlist}
\item[\hskip \labelsep{\bf Proof of #1:}]}{\foorp \end{trivlist}}
\newcommand{\foorp}{{\unskip\nobreak\hfil\penalty50
 \hskip1em\vadjust{}\nobreak\hfil \vrule height5pt width5pt depth0pt
 \parfillskip=0pt \finalhyphendemerits=0 \par}}
\newcommand{\Res}{\mathop{\text{\rm Res}}\nolimits}
\newcommand{\Ind}{\mathop{\text{\rm Ind}}\nolimits}
\newcommand{\Hom}{\mathop{\text{\rm Hom}}\nolimits}
\newcommand{\Span}{\mathop{\text{\rm span}}\nolimits}

\newcommand{\vyp}[3]{\oldstylenums{\bfseries #1}
(\oldstylenums{#2}),~\oldstylenums{#3}}
\newcommand{\vypp}[4]{\vyp{#1}{#2}{#3}--\oldstylenums{#4}}
\newcommand{\svypp}[5]{(\oldstylenums{#1}) \vypp{#2}{#3}{#4}{#5}}
\newcommand{\Dr}[1]{D_{#1}}
\newcommand{\DrV}[2]{\Dr{#1} (#2)}
\newcommand{\CnV}[1]{\Delta (#1)}
\newcommand{\DTs}[2]{D_{#1, #2}}
\newcommand{\Sym}[1]{S (#1)}
\newcommand{\Rn}[1]{R_{#1}}

\newcommand{\GL}{\textsl{GL}}
\newcommand{\kat}[1]{\ka_{#1}}
\newcommand{\xtil}{\tilde{x}}
\newcommand{\chihat}{\hat{\chi}}
\newcommand{\gamhat}{\hat{\gamma}}
\newcommand{\ctwo}[3][p]{#2 {#3}^{#1} - {#2}^{#1} #3}
\newcommand{\cS}{\Z}
\newcommand{\cT}{T}
\newcommand{\bL}{b_L}
\newcommand{\astar}{A^*}
\newcommand{\bstar}{B^*}
\newcommand{\pgp}{p^{1+2n}_+}

\title{Chern Classes and Extraspecial Groups}
\author{David J. Green\thanks{The support of the Deutsche
Forschungsgemeinschaft and the Max Kade Foundation is gratefully
acknowledged, as is the kind hospitality of the ETH Z\"urich.  Present address:
Dept of Math., Univ.~of Chicago,
5734 S.~University, Chicago IL60637, USA.} \\
Inst.\ f.\ Experimentelle Math. \\ Ellernstr.~29 \\ D--45326 Essen \\
\textsc{Germany} \\ \texttt{\small djg@math.uchicago.edu}
\and Ian J. Leary\thanks{The support of the Eidgen\"ossische Technische
Hochschule in Z\"urich and of the Leibniz Fellowship Programme is acknowledged
with thanks.} \\
Max-Planck-Inst.\ f.\ Math. \\
Gottfried-Claren-Str.~26 \\ D--53225 Bonn \\ \textsc{Germany} \\
\texttt{\small leary@mpim-bonn.mpg.de}}
\date{\makedate12/5/95/}

\begin{document}

\maketitle


\begin{abstract}
The mod-$p$ cohomology ring of the extraspecial $p$-group of exponent~$p$ is
studied for odd~$p$.  We investigate the subquotient~$\ch{G}$ generated by
Chern classes modulo the nilradical.  The subring of~$\ch{G}$ generated by
Chern classes of one-dimensional representations was studied by Tezuka and
Yagita.  The subring generated by the Chern classes of the faithful irreducible
representations is a polynomial algebra.  We study the interplay
between these two families of generators, and obtain some relations between
them.
\end{abstract}

\paragraph{Introduction}
One of the major outstanding problems in the cohomology of finite groups
is the determination of the cohomology rings of the extraspecial groups.
The case of extraspecial $2$-groups was solved elegantly and completely
by Quillen (see~\cite{Quillen:mod-2}), and there has been much work since
then on the extraspecial $p$-groups for odd~$p$.  On the one hand, the
cohomology ring of the extraspecial $p$-groups of order~$p^3$ and
exponent~$p$ has been determined by Lewis for integral coefficients
and by the second named author for mod-$p$ coefficients
(see \cite{Lewis},~\cite{Leary:mod-p})\@.
On the other hand, there have been major advances in the
general problem, which have concentrated on calculating the cohomology ring
modulo its nilradical.  Tezuka and Yagita calculated this up to
inseparable isogeny by a generalization of Quillen's methods
(see~\cite{Tezuka-Yagita:varieties}), and Benson and Carlson have collected
and refined the knowledge to date in their expository
paper~\cite{Benson-Carlson}.

In this paper we consider the problem of determining the cohomology ring
modulo its nilradical exactly, rather than up to inseparable isogeny.  We
shall attempt this by studying Chern classes of the irreducible representations
of the extraspecial group~$\ex{2n}$ which are obtained by inducing from
maximal abelian subgroups.
Chern classes will not in general generate the cohomology ring modulo its
nilradical, even for $p$-groups (see the paper~\cite{Leary-Yagita} of
Yagita and the second named author, where examples of order~$p^4$ and
rank~$2$ are given).
However, this approach does indeed give us new cohomology classes.

This is clearly seen by considering the group of order~$p^3$.
The cohomology ring modulo its radical of the extraspecial group~$\ex2$
is the quotient of the polynomial ring
$\f \lbrack \alpha, \beta, \kat{0}, \zeta \rbrack$ by the relations 
$\alpha^p \beta - \alpha \beta^p = 0$, $\alpha \kat{0} = - \alpha^p$,
$\beta \kat{0} = - \beta^p$ and
$\kat{0}^2 = \alpha^{2p-2} - \alpha^{p-1} \beta^{p-1} + \beta^{2p-2}$.
Here $\alpha$~and $\beta$ are first Chern classes of degree~$1$
representations, whereas $\kat{0}$~and $\zeta$ are Chern classes of a
degree~$p$ irreducible representation.
However, the best result known to date for general extraspecial groups only
says that this is up to inseparable isogeny the quotient of
$\f \lbrack \alpha, \beta, \zeta \rbrack$ by the one relation
$\alpha^p \beta - \alpha \beta^p = 0$.

It is very probable that the full description of the mod-$p$ cohomology
ring is exceedingly complicated: one just needs to look at the result
for the group of order~$p^3$ to get an idea of this.  It is not even certain
that it is practical to calculate the whole of the cohomology ring modulo
its radical.  But the example of the~$p^3$ case suggests that understanding
the Chern classes will be a major step in the right direction.
When investigating the cohomology ring of a $p$-group, the Chern subring
modulo nilradical is therefore a worthy and interesting object of study.

The outline of the paper is as follows.
Let~$p$ be an odd prime, and $n \geq 1$.
After recalling necessary information about group cohomology, Chern classes,
extraspecial $p$-groups and Dickson invariants, we obtain generators for the
Chern subring modulo nilradical in Proposition~\ref{prop:res_kz}.
For~$\pgp$, there are $3n+1$ generators: $\alpha_i$~and $\beta_i$ for
$1 \leq i \leq n$; $\kat{r}$~for $0 \leq r \leq n-1$; and~$\zeta$.
The~$\kat{r}$ are the new generators: the other elements generate the
subring studied by Tezuka and Yagita.
The~$\kat{r}$ are Chern classes of
a degree~$p^n$ faithful irreducible representation of~$\pgp$,
and restrict to \maxel.s as Dickson invariants.

Our aim is to understand the relationship between the new generators and
the old.  In Proposition~\ref{prop:kappa_n}, we obtain an elegant
alternating sum formula expressing~$\kat{0}^{2^n}$ as a polynomial in the
$\alpha_i$~and $\beta_i$.  Theorem~\ref{thm:kappa_integral}, the main
result of the paper, generalises this by showing that $\kat{r}^{2^{n-r} p^r}$
also lies in the subring~$\TY$ generated by the $\alpha_i$~and $\beta_i$.
The idea behind both these results is as follows.  Pick a \maxel.
$M$~of $\pgp$; find some expression (typically a Dickson invariant)
in the $\alpha_i$~and $\beta_i$ whose restriction to~$M$ equals that
of~$\kat{r}$; patch these approximations together; and then appeal to Quillen's
theorem, which states that the \maxel.s detect non-nilpotent elements.

In a separate paper~\cite{Green}, the first named author completes this project
for $n=2$ by obtaining a presentation for the Chern subring modulo nilradical
of~$p^{1+4}_+$.  In particular, it is shown that for all $n \geq 1$, the
Chern class~$\kat{0}$ lies
outside the subring studied by Tezuka and Yagita.


\paragraph{A theorem of Quillen}
In this paper we study cohomology rings modulo their nilradicals.  The
rationale for this is provided by the following theorem
of Quillen.

\begin{theorem}
\label{thm:Quillen_original}
{\em (Quillen~\cite{Quillen},\cite{Quillen-Venkov})}
Let~$G$ be a finite group, $p$~a prime number, and~$k$ a field of
characteristic~$p$.
Then a class $\xi \in \coho*Gk$ is nilpotent if and only if its restriction
to every \elems{p} of~$G$ is nilpotent.
\foorp \end{theorem}

Let~$G$ be a finite group, and $p$~a prime number.  Define $\h*{G, \f}$
to be the quotient of the graded commutative ring $\hf*G$ by its
nilradical.  If the value of~$p$ is clear from the context, we will just
write~$\h*G$.

Of course, $\h*G$~is strictly commutative.
If $\phi \colon H \rightarrow G$ is a group homomorphism, then nilpotent
classes in $\hf*G$ are mapped under~$\phi^*$ to nilpotent classes
in $\hf*H$: hence~$\phi$ induces a well-defined ring homomorphism
$\phi^* \colon \h*G \rightarrow \h*H$.  In particular, there is
a well-defined restriction homomorphism if~$H$ is a subgroup of~$G$.\@
The version of Quillen's theorem that we shall use in this paper is
now a trivial corollary of Theorem~\ref{thm:Quillen_original}.

\begin{coroll}
Let~$G$ be a finite group; let~$p$ be a prime number; and let~$\xi$
be a class in~$\h*G$.  Then~$\xi$ is zero
if and only if $\Res_A \xi = 0$ in~$\h*A$ for every \elems{p} $A$~of $G$.
\foorp
\end{coroll}

\paragraph{Chern classes}
For a concise introduction to Chern classes of group representations, we
refer the reader to the appendix of~\cite{Atiyah}.
Although Chern classes strictly belong to $\hz{G}$, they can be considered as
elements of~$\h*G$ via the map
$\hz{G} \rightarrow \coho*G{\f} \rightarrow \h*G$.
Write~$\ch{G}$ for the subring of~$\h*G$ generated by Chern classes.
This algebra~$\ch{G}$ is a large subquotient of $\hf*G$;
any pair of elements may be compared in a straightforward way;
and the Chern classes of the irreducible representations form a finite set of
generators (by the Whitney sum formula).
This makes the algebra~$\ch{G}$ an object worthy of study.
The object of this paper is to investigate it
for $G$~extraspecial.

Because their restrictions can be calculated directly, Chern classes lend
themselves particularily well to a study using Quillen's theorem.  Taking
the Chern classes of a representation 
commutes with taking its restriction to a subgroup, and the Whitney sum
formula expresses the Chern classes of a direct sum of representations in terms
of the Chern classes of the summands.  After restricting a representation to
an abelian subgroup, the irreducible summands all have degree one;
and the Chern classes of such representations are very well understood.

The mod-$p$ cohomology rings of the \elem{p}s are well known, and may easily
be derived from the cohomology ring of the cyclic group using the \Kuenneth.
Theorem.
Let~$A$ be an \elem{p} of $p$-rank~$m$.  Then~$A$ is also an $m$-dimensional
$\f$-vector space.  Embed the additive group of $\f$~in $\mult{\C}$ by
sending $1$~to $\exp (2 \pi i / p)$.
This induces an isomorphism between $\Hom (A, \mult{\C})$, a group under
tensor product, and the dual vector space~$A^*$, a group under addition.
There is therefore a map $c_1 \colon A^* \rightarrow \h2A$ which sends an
element of~$A^*$ to the first Chern class of the corresponding one-dimensional
representation.  The induced map from the symmetric algebra $\Sym{A^*}$ to
$\h*A$ is an isomorphism of $\f$-algebras.  So~$\h*A$ is an integral domain,
and in fact a polynomial algebra.

\paragraph{Extraspecial Groups}
From now on we fix an odd prime~$p$, and denote by~$G$ the \extr{p} of
order~$p^{2 n + 1}$ and exponent~$p$.  There is a central extension
$1 \rightarrow N \rightarrow G \stackrel{\psi}{\rightarrow} E \rightarrow 1$
with $E$~an \elem{p} of $p$-rank~$2n$, and $N$ cyclic of order~$p$.
Identify $N$~with $\f$, and view~$E$ as a $2n$-dimensional $\f$-vector space.
There is a well-defined nondegenerate symplectic bilinear form
$b \colon E \times E \rightarrow N$ given by
$b (x_1, x_2) = [\xtil_1, \xtil_2]$
for any~$\xtil_1$, $\xtil_2 \in G$ such that $\psi (\xtil_i) = x_i$,
where $[a, b] = a b a^{-1} b^{-1}$.
There is a natural bijection between the set of \maxel.s $M$~of $G$ and
the set of \mtis.s $I$~of $E$, given by $I = M / N$ and $M = \psi^{-1} (I)$.
Every~$M$ has $p$-rank $n + 1$, and every~$I$ has dimension~$n$.

To determine the irreducible characters of~$G$, embed~$\f$ in~$\mult{\C}$
as before.
Let~$\chihat$ be a nontrivial linear character of~$N$, and define
$\chihat_j = \chihat^{\otimes j}$ for $1 \leq j \leq p-1$.

\begin{lemma}
\label{lem:irred_chars}
There are~$p^{2n}$ linear characters of~$G$, and all factor
through~$\psi$.  These correspond to the elements of~$E^*$.
The $p - 1$ remaining irreducible characters all have
degree~$p^n$ and are induced from any \maxel. of~$G$.  They may be labelled
$\chi_1$, \dots,~$\chi_{p-1}$ such that for every $1 \leq j \leq p-1$ and
every $g \in G$,
\begin{equation}
\label{eqn:induced_chars}
\chi_j (g) = \left\{ \begin{tabular}{ll}
$p^n \hat{\chi}_j (g)$ & if $g \in N$, and \\
0                      & otherwise. \end{tabular} \right.
\end{equation}
It follows that, for any \maxel. $M$~of $G$, the restriction~$\Res_M (\chi_j)$
is the sum of all those linear characters of~$M$ whose restriction to~$N$
is~$\hat{\chi}_j$.
\end{lemma}

\begin{proof}
Choose a \maxel. $M$~of $G$.  For each $1 \leq j \leq p-1$, pick a
linear character~$\hat{\chi}'_j$ of~$M$ whose restriction to~$N$
is~$\hat{\chi}_j$.  Define the character~$\chi_j$ of~$G$ to
be~$\Ind^G (\hat{\chi}'_j)$.  The character formula for induced characters then
shows that~(\ref{eqn:induced_chars}) is satisfied.
Using the orthogonality relations, and summing squares of degrees, it is seen
that we have all irreducible characters.
\end{proof}

\paragraph{Dickson invariants}
We will see in Proposition~\ref{prop:res_kz} that the Chern classes of the
induced representations restrict to \maxel.s as Dickson invariants.  We
now recall the salient facts about these invariants.
For a proof, see Benson's book~\cite{Benson:poly_invts}.

\begin{theorem}
\label{thm:Dickson}
{\em (Dickson)}
Let~$V$ be a finite dimensional $\f$-vector space, and let $m = \dim (V)$.
\begin{enumerate}
\item
\label{enum:DicksonInvtsExist}
There exist {\em Dickson invariants} $\DrV{0}{V}$, \dots, $\DrV{m-1}{V}$ in
the symmetric algebra~$\Sym{V}$, with $\DrV{r}{V}$~in $S^{p^m - p^{r}} (V)$ for
each $0 \leq r \leq m-1$, such that
\begin{equation}
\label{eqn:Dickson}
\prod_{v \in V} (X - v) =
X^{p^m} + \sum_{r = 0}^{m-1} (-1)^{m-r} \DrV{r}{V} X^{p^r} \; .
\end{equation}
\item
\label{enum:GLV_Invts}
The Dickson invariants are algebraically independent.  The
ring of invariants $\Sym{V}^{\GL (V)}$ is the polynomial
algebra $\f \lbrack \DrV{0}{V}, \ldots, \DrV{m-1}{V} \rbrack$.
%
\foorp
\end{enumerate}
\end{theorem}

In the literature, $\DrV{r}{V}$~is usually denoted~$c_{m, r}$.  New notation
is introduced here in order to identify the vector space~$V$ explicitly, and
to avoid a clash with the notation $c_r (\rho)$ for Chern classes.

We now describe the relationship with Dickson invariants of quotient spaces.
This takes a particularily elegant form when we work with dual spaces.

\begin{lemma}
\label{lem:Dickson_larger}
Let~$V$ be an $m$-dimensional $\f$-vector space, and~$U$ an
$\ell$-codi\-men\-sional subspace.  The inclusion of $U$~in $V$ induces a
restriction map $\Sym{V^*} \rightarrow \Sym{U^*}$ and, for every
$0 \leq r \leq m-1$,
\begin{equation}
\Res_U \left( \DrV{r}{V^*} \right) = \left\{ \begin{tabular}{l@{ }l}
$\DrV{r - \ell}{U^*}^{p^{\ell}}$ & if $\ell \leq r$, and \\
$0$ & otherwise.
\end{tabular} \right.
\end{equation}
\end{lemma}

\begin{proof}
Obvious from the definition of the Dickson invariants.
\end{proof}
The top Dickson invariant~$\DrV{0}{V}$ is, up to a sign, the product of all
nonzero elements of~$V$.  A generalisation (due to Macdonald) of this
interpretation will play an important role in this paper.

\begin{theorem}
\label{thm:Macdonald}
{\em (Macdonald)}
Let~$V$ be an $m$-dimensional $\f$-vector space, and let $0 \leq r \leq m-1$.
Denote by~$\V_r$ the set of all $r$-dimensional subspaces of~$V$.
For any $Y \in \V_r$, write~$P_{V,Y}$ for the product
$\prod_{v \in V \setminus Y} v$ in the symmetric algebra~$\Sym{V}$.
Then
\begin{equation}
\label{eqn:Macdonald}
\DrV{r}{V} = (-1)^{m-r} \sum_{Y \in \V_r} P_{V, Y} \; .
\end{equation}
\end{theorem}

\begin{proof}
(\cf \cite{Benson:poly_invts}, p.~\oldstylenums{91})\@
The right hand side of~(\ref{eqn:Macdonald}) is clearly an invariant
of~$\GL(V)$, and so must be a scalar multiple of the left hand side.
Pick any $Y \in \V_r$, and project down onto~$\Sym{U}$, where $U = V / Y$.
This sends the left hand side to~$\DrV{0}{U}^{p^r}$.
On the right hand side, all summands are sent to zero,
except that~$P_{V,Y}$ is sent
to~$P_{U,0}^{p^r}$.  Therefore we need only establish the case $r=0$: but this
is an immediate consequence of Theorem~\ref{thm:Dickson}.
\end{proof}

\paragraph{Chern classes for extraspecial groups}
We start by considering the degree one representations.
Choose a symplectic basis~$A_1$, \dots, $A_n$,~$B_1$, \dots, $B_n$~for $E$.
That is, $A_i \perp A_j$, $B_i \perp B_j$ and $b (A_i, B_j) = \delta_{ij}$.
Take the corresponding dual basis~$\astar_1$, \dots, $\bstar_n$~for $E^*$,
and recall that one-dimensional representations of~$G$ are identified with
elements of~$E^*$.  For $1 \leq i \leq n$, define $\alpha_i = c_1 (\astar_i)$
and $\beta_i = c_1 (\bstar_i)$.  Equivalently,
consider the $\astar_i$~and $\bstar_i$ as elements of~$\h2E$ via the
isomorphism $\h*E \cong \Sym{E}$, and define $\alpha_i$,~$\beta_i$ to be
the inflations $\psi^*(\astar_i)$,~$\psi^*(\bstar_i)$ respectively.

Let~$\rho_1$ be a representation of~$G$ affording the induced
character~$\chi_1$.  Define $\kat{r} = (-1)^{n-r} c_{p^n - p^r} (\rho_1)$ for
$0 \leq r \leq n-1$, and $\zeta = c_{p^n} (\rho_1)$.
Let~$\gamma$ be the element of~$N^*$ corresponding to the nontrivial linear
character $\chihat$~of $N$.

\begin{prop}
\label{prop:res_kz}
The Chern subring of~$\h*{\pgp}$ is generated by~$\kat{0}$, \dots, $\kat{n-1}$,
$\zeta$, $\alpha_1$, $\beta_1$, \dots, $\alpha_n$~and $\beta_n$.
Let~$M$ be a \maxel. of~$\pgp$,
and $I = M / N$ the corresponding \mtis. of~$E$.
Notice that $\Sym{M^*} \cong \Sym{I^*} \otimes_{\f} \f \lbrack \gamma \rbrack$.
We have
\begin{equation}
\Res_M \kat{r}  = \DrV{r}{I^*} \quad \text{for $0 \leq r \leq n-1$, and} \quad
\Res_N \zeta = \gamma^{p^n} \; .
\end{equation}
\end{prop}

\begin{proof}
The $\alpha_j$~and $\beta_j$ arise from a basis for~$E^*$, and so the
first Chern class of every degree one representation is in their span.

The inflation map $\h2I \rightarrow \h2M$ is the inclusion of $I^*$~in $M^*$,
and the restriction map $\h2M \rightarrow \h2N$ is the projection of $M^*$~onto
$N^*$ with kernel~$I^*$.  Let~$\gamhat \in M^*$ be a representative of the
coset whose restriction to $N^*$~is $\gamma$.
It follows from Lemma~\ref{lem:irred_chars} that, for every
$1 \leq j \leq p-1$, the induced character~$\chi_j$ restricts to~$M$ as the
direct sum of the elements of the coset $j \gamhat + I^*$, considered as
linear characters of~$M$.
Let~$\rho_j$ be a representation of~$G$ affording~$\chi_j$.
Using the Whitney sum formula, the total Chern class of~$\rho_j$ is
\begin{align*}
\Res_M c(\rho_j) & = \prod_{v \in I^*} (1 + j \gamhat + v) \\
& = 1 + \sum_{r=0}^{n-1} (-1)^{n-r} \DrV{r}{I^*}
+ j  \left( \gamhat^{p^n} + \sum_{r=0}^{n-1} (-1)^{n-r} \DrV{r}{I^*}
\gamhat^{p^r} \right) \; .
\end{align*}
Therefore $c(\rho_j) = 1 + \kat{0} + \cdots + \kat{n-1} + j \zeta$, and the
restrictions are as claimed.
\end{proof}

\paragraph{Remark:}
An argument involving the Adams--Frobenius operations may 
be used to show that the
equalities $c_s (\rho_j)  =  j^s c_s(\rho_1)$ for all $s \geq 1$ hold
even in $\hz{G}$.

\paragraph{Important subrings}
There are two important subrings of~$\ch{G}$.  Both of these have the same
Krull dimension as~$\ch{G}$ itself, and together they generate~$\ch{G}$.
The structure of each of these subrings is known; but understanding
how elements of one ring relate to elements of the other is more complicated.
The core of this paper is an attempt to start understanding this
relationship.

The first subring is generated by~$\zeta$, the~$\alpha_i$ and the~$\beta_i$,
and was studied by Tezuka and Yagita.
The Chern classes which generate~$\ch{G}$ were originally defined in $\hz{G}$,
and therefore correspond to well-defined elements of $\hf*G$.
Let~$\TY$ denote the subring of~$\h*G$ generated by the $\alpha_i$~and
$\beta_i$.  For $r \geq 1$, let $\Rn{r} = \ctwo[p^r]{\alpha_1}{\beta_1}
+ \cdots + \ctwo[p^r]{\alpha_n}{\beta_n}$.

\begin{theorem}
\label{thm:Tezuka-Yagita}
{\em (Tezuka--Yagita~\cite{Tezuka-Yagita:varieties})}
The subalgebra $\TY$~of $\h*G$ is the quotient of the polynomial
algebra on the $\alpha_i$~and $\beta_i$ by the ideal generated by
the~$\Rn{r}$ for $1 \leq r \leq n$.  Moreover, $\Rn{r} = 0$ in $\h*G$ for
all $r \geq 1$.  The ring $\TY \otimes_{\f} \f \lbrack \zeta \rbrack$ is
in fact a subalgebra of $\hf*G$, and every element of $\hf*G$ has some power
lying in this subalgebra.
\foorp
\end{theorem}

The second large subring is generated by~$\zeta$ and the~$\kat{r}$.

\begin{prop}
\label{prop:no_zero_divisors}
The Chern classes $\kat{0}$, \dots, $\kat{n-1}$ and $\zeta$ are algebraically
independent over~$\f$.  Moreover, no polynomial in these elements is a zero
divisor in~$\h*G$.
The $\f$-algebra $\h*G$ is finite
over the subalgebra generated by the $\kat{r}$~and $\zeta$.
\end{prop}

\begin{proof}
Algebraic independence is a result of the algebraic independence of the
Dickson invariants.
By \Quillen., every non-zero element of~$\h*G$ has non-zero
restriction to~$\h*M$ for some \maxel.~$M$.  By Proposition~\ref{prop:res_kz},
restriction to~$M$ is an injection on the subring we consider.  But~$\h*M$
is an integral domain.
For the last part we appeal to Venkov's proof of the Evens--Venkov theorem,
since~$\rho_1$ is faithful.
\end{proof}

That the first of these subrings is not contained in the second is clear
by degree considerations.  Conversely, $\kat{0}$ lies in the Tezuka--Yagita
subring for no value of~$n$.  For $n=1$ this can be seen from Lewis'
paper~\cite{Lewis}, and it is proved in~\cite{Green} for general~$n$.

We start our investigation of the relationship between these two subrings
by establishing one simple identity.

\begin{lemma}
\label{lem:newrel}
Let~$x$ be one of the $\alpha_i$~or $\beta_i$; more generally, let~$x$ be the
first Chern class of a one-dimensional representation of~$G$. Then
\begin{equation}
\label{eqn:newrel}
x^{p^n} - x^{p^{n-1}} \kat{n-1} + \cdots + (-1)^r x^{p^{n-r}} \kat{n-r}
+ \cdots + (-1)^n x \kat{0} = 0 \; .
\end{equation}
\end{lemma}

\begin{proof}
Let~$M$ be a \maxel. of~$G$, and $I$~the corresponding \mtis. of~$E$.
Then $\Res_M (x) \in I^*$, and so
$\prod_{v \in I^*} (\Res_M (x) - v) = 0$.
But, from the definition of the~$\kat{r}$, this product equals the
restriction to~$M$ of the left hand side of~(\ref{eqn:newrel}).  The result
follows by \Quillen..
\end{proof}

We would like to obtain all the relations between the generators of~$\ch{G}$.
In this paper we begin this task by investigating which powers of~$\kat{r}$
lie in the Tezuka--Yagita ring.

\paragraph{Bases}
In this section, we establish a result about the restrictions
of the $\astar_j$~and $\bstar_j$
to the dual space~$I^*$, for any \mtis. $I$~of $E$.
To this end,
we shall introduce a symplectic form~$\bL$ on~$E^*$, which will also play an
important role in subsequent sections.
For every $0 \leq r \leq n$, let $\Z_r$ denote the set of all
$r$-dimensional subspaces $V$~of $E^*$ which have a basis~$y_1$, \dots,~$y_r$
in which each~$y_i$ is either $\astar_i$~or $\bstar_i$.

\begin{prop}
\label{prop:extendZ_r}
Let $1 \leq r \leq n$, and let~$I$ be a \mtis. of~$E$.
Let $V \in \Z_{r-1}$, and suppose that the restriction of $V$~to $I^*$
also has dimension~$r-1$.
Then there exists an element of~$\Z_r$ which contains~$V$, and whose
restriction to~$I^*$ has dimension~$r$.

Hence, for every $0 \leq r \leq n$ and for every~$I$, there is at least one
$V \in \Z_r$ whose restriction to~$I^*$ has dimension~$r$.
\end{prop}


The nondegenerate symplectic form $b$~on $E$ induces an $\f$-vector space
isomorphism $L \colon E \rightarrow E^*$ as follows: for all~$e$, $e' \in E$,
$L (e) (e') = b (e, e')$.  There is then a unique symplectic form~$\bL$
on~$E^*$ such that
$\bL \left(L(e), L(e') \right) = b (e, e')$,
for all~$e$, $e' \in E$.
Since~$b$ is nondegenerate, so is~$\bL$.

\begin{lemma}
\label{lem:extendZ_r}
Let $U$~and $V$ be subspaces of~$E^*$, and let~$I$ be a \ti. subspace of~$E$.
If $U \perp V$, if $V$~is \ti., and if $U \subseteq V + L (I)$, then
$U$~is \ti..
\end{lemma}

\begin{proof}
Let~$u$, $u' \in U$.
Since $U \subseteq V + L (I)$, there exist $v$, $v' \in V$ and $i$, $i' \in I$
such that $u = v + L (i)$ and $u' = v' + L (i')$.  Then
$\bL (u, u') = \bL (v, u') + \bL (u, v') + \bL (L(i), L(i')) - \bL (v, v')$,
and each of these terms is zero by assumption.
\end{proof}

\begin{proofof}{Proposition~\ref{prop:extendZ_r}}
Let $U = \Span (\text{$\astar_r$, $\bstar_r$})$.  Then
$V$~is totally iso\-tropic and $U \perp V$, but $U$~is not \ti..
Hence, by the Lemma, $U$~is not contained in $V + L (I)$.  In particular,
at least one of $\astar_r$,~$\bstar_r$ does not lie in $V + L(I)$.  But
$u \in V + L(I)$ if and only if the restriction of~$u$ to~$I^*$ lies in the
restriction of~$V$.  The last part follows by induction on~$r$.
\end{proofof}

\paragraph{Characteristic functions}
In this section we show that~$\kat{0}^{2^n}$ lies in the Tezuka--Yagita
subring~$\TY$.  This is a special case of Theorem~\ref{thm:kappa_integral},
and the results of this section are not necessary to prove that theorem.
However, the methods we use are more transparent here than in the general case,
and we also succeed in establishing an elegant formula for~$\kat{0}^{2^n}$.

We shall prove that $\kat{0}^{2^n} \in \TY$ using characteristic functions,
analogously to the alternating sum formula for the measure of a finite union.
Of fundamental importance is the following special case of
Lemma~\ref{lem:Dickson_larger}.
Let~$V$ be an $n$-dimensional subspace of~$E^*$, let~$M$ be a \maxel. of~$G$,
and let~$I$ be the corresponding \mtis.  of~$E$.  Then
\begin{equation}
\Res_M (\DrV{0}{V}) = \left\{ \begin{array}{l@{\quad}l}
\DrV{0}{I^*} & \text{if $\Res_I (V)$ is the whole of~$I^*$, and} \\
0 & \text{otherwise.}
\end{array} \right.
\end{equation}

Recall that~$\Z_n$ denotes the set of all $n$-dimensional subspaces
of~$E^*$ which have a basis of the form~$y_1$, \dots,~$y_n$ such that
each~$y_i$ is either $\astar_i$~or $\bstar_i$.
In this section we will write $\cS$~for $\Z_n$.
Note that~$\cS$ has cardinality~$2^n$.

We now introduce some more notation for this section.
Let~$\I$ denote the set of all \mtis.s of~$E$.
For each subset $\cT$~of $\cS$, define $\IT{\cT}$ to be the set of all
$I \in \I$ such that $\Res_I (\x) = I^*$ for every $\x \in \cT$.
If $\x \in \cS$, write~$\IT{\x}$ for~$\IT{\{\x\}}$.

For any subset $\cT$~of $\cS$, define
$\chi_{\cT} \colon \I \rightarrow \{ \text{$0$, $1$} \}$
to be the characteristic function of~$\IT{\cT}$: that is, for $I \in \I$,
\begin{equation}
\label{eqn:chi_T}
\chi_{\cT} (I) = \left\{ \begin{tabular}{l@{ }l}
$1$ & if $I \in \IT{\cT}$, and \\
$0$ & otherwise.
\end{tabular} \right.
\end{equation}
For $\x \in \cS$, write~$\chi_{\x}$ for~$\chi_{\{\x\}}$.
The following result is now a consequence of Quillen's Theorem.

\begin{lemma}
\label{lem:chi_T}
Let $s \geq 1$; let~$\cT$ be a non-empty subset of~$\cS$ such that
$| \cT | \leq s$;
and let~$\x_1$, \dots,~$\x_s$ be a sequence of elements
of~$\cT$ in which each element of~$\cT$ appears at least once.
Define
\begin{equation}
\DTs{\cT}{s} = \psi^* \left( \prod_{j=1}^s \DrV{0}{\x_j} \right)
\quad \text{in $\h{2 s (p^n - 1)}G$.}
\end{equation}
Let~$M$ be a \maxel. of~$G$, and let~$I$ be the corresponding \mtis. of~$E$.
Then
\begin{equation}
\label{eqn:indicator_T}
\Res_M (\DTs{\cT}{s}) = \chi_{\cT} (I) \DrV{0}{I^*}^s \; ,
\end{equation}
and so $\DTs{\cT}{s}$ is independent of the choice of the~$\x_j$,
which justifies the notation. \foorp
\end{lemma}

\begin{lemma}
\label{lem:MT_properties}
Let $\cT_1$~and $\cT_2$ be subsets of~$\cS$.  Then
\begin{enumerate}
\item \label{enum:cover}
$\displaystyle\bigcup_{\x \in \cS} \IT{\x} = \I$.
\item \label{enum:intersection}
$\IT{\cT_1} \cap \IT{\cT_2} = \IT{\cT_1 \cup \cT_2}$, and therefore
$\chi_{\cT_1 \cup \cT_2} = \chi_{\cT_1} \chi_{\cT_2}$.
\item \label{enum:union}
$\displaystyle\prod_{\x \in \cS} (1 - \chi_{\x}) = 0$.
\end{enumerate}
\end{lemma}

\begin{proof}
Part~\ref{enum:cover} is a consequence of Proposition~\ref{prop:extendZ_r}.
Part~\ref{enum:intersection} is an immediate consequence of the definition.
Part~\ref{enum:union} now follows from the formula for the characteristic
function of a union.
\end{proof}

\begin{prop}
\label{prop:kappa_n}
Let $s \geq 2^n$.  Then
\begin{equation}
\label{eqn:kappa_n}
\kat{0}^s =
- \sum_{\emptyset \not = \cT \subseteq \cS} (-1)^{|\cT|} \DTs{\cT}{s} .
\end{equation}
\end{prop}

\begin{proof}
We of course use Quillen's Theorem.  Let~$M$ be a \maxel., and~$I$ the
associated \mtis..
By Proposition~\ref{prop:res_kz},
$\Res_M (\kat{0}^s) = \chi_{\emptyset} (I) \DrV{0}{I^*}^s$.
Now, by Lemma~\ref{lem:chi_T} and Lemma~\ref{lem:MT_properties},
\begin{align}
\Res_M \left( \kat{0} +
\sum_{\emptyset \not = \cT \subseteq \cS} (-1)^{|\cT|} \DTs{\cT}{s} \right)
& = \sum_{\cT \subseteq \cS} (-1)^{|\cT|} \chi_{\cT} (I) \DrV{0}{I^*}^s \\
& = \left( \prod_{\x \in \cS} \left(1 - \chi_{\x} (I) \right) \right)
\DrV{0}{I^*}^s \\
& = 0 \; ,
\end{align}
which proves the result.
\end{proof}

This result does not imply that $\kat{0}^s \not \in \TY$ whenever $s < 2^n$.
In fact, it is proved in~\cite{Green} that, for $n = 2$, $\kat{0}^s \in \TY$
if and only if $s \geq 2$.
However, the following lemma demonstrates that~$\kat{0}^s$ cannot be
expressed in terms of the~$\DTs{\cT}{s}$ if $s < 2^n$.

\begin{lemma}
\label{lem:MT_inequalities}
Let $\x \in \cS$, and let $\cT_1$~and $\cT_2$ be subsets of~$\cS$.  Then
\begin{enumerate}
\item \label{enum:nonempty}
$\IT{\cS}$ is not empty.
\item \label{enum:detect_x}
$\IT{\cS \setminus \x}$ is strictly larger than~$\IT{\cS}$.
\item \label{enum:equality}
$\IT{\cT_1} = \IT{\cT_2}$ if and only if $\cT_1 = \cT_2$.
\end{enumerate}
\end{lemma}

\begin{proof}
For each $1 \leq i \leq n$, define $X_i = B_i A_i \in E$; then~$X_1$,
\dots,~$X_n$ generate a \mtis. which lies
in~$\IT{\cS}$.
The automorphism group of~$G$ acts transitively on the set of \maxel.s,
and so in part~\ref{enum:detect_x} we may assume without loss of generality
that $\x = \{ \text{$\astar_1$, \dots, $\astar_n$} \}$.
Define elements~$Y_1$, \dots,~$Y_n$ of~$E$ by
$Y_r = B_1 \ldots B_r A_r A_{r+1}^{-1}$ if $r \leq n-1$,
and $Y_n = B_1^{2-n} \ldots B_r^{r+1-n} \ldots B_n A_n A_1^{-1}$.
Then~$Y_1$, \dots,~$Y_n$ commute with each other, and generate a  \mtis.
which lies in~$\IT{\cS \setminus \x}$ but not
in~$\IT{\cS}$.
Finally, part~\ref{enum:equality} now follows.
\end{proof}

\paragraph{Integrality}
Each~$\kat{r}$ is integral over~$\TY$, by the Tezuka--Yagita theorem.
In this section, we obtain explicit monic polynomial equations satisfied
by the~$\kat{r}$, and prove that~$\kat{r}^{t p^r}$ lies in~$\TY$ for
sufficiently large~$t$.
Recall that~$E^*$ carries a nondegenerate symplectic form~$\bL$.

\begin{lemma}
\label{lem:Estar_correspondence}
Let~$V$ be a subspace of~$E^*$, and~$I$ a \mtis. of~$E$\@.
Suppose that $\dim \Res_I (V) = \dim (V)$.  Then
$\Res_I (V^{\perp}) = I^*$.
\end{lemma}

\begin{proof}
Let $A$ be the subspace $L^{-1} (V)$ of~$E$.
Then $\dim \Res_I (V) = \dim (V)$ if and only if $I \cap A = 0$,
and $\Res_I (V) = I^*$ if and only if $A + I = E$.
So we must prove that $A^{\perp} + I = E$ if $I \cap A = 0$.
Since~$b$ is nondegenerate and $I^{\perp} = I$, this is standard linear
algebra.
\end{proof}

Let~$V$ be a finite-dimensional $\f$-vector space.  Choose one non-zero vector
from each one-dimensional subspace of~$V$, and define $\CnV{V} \in \Sym{V}$
to be the product of all these subspace representatives.
Then~$\CnV{V}$ is well-defined up to multiplication by a scalar, and
Macdonald's Theorem shows us that $\CnV{V}^{p-1} = \DrV{0}{V}$.

Recall that the inflation map~$\psi^*$ induces an $\Sym{E^*}$-module structure
on~$\h*G$.

\begin{prop}
\label{prop:K_rD_r}
Let $0 \leq r \leq n-1$; let $n-r \leq s \leq n$; and let~$V$ be an
$s$-dimensional subspace of~$E^*$.  Then
$\CnV{V} \left( \kat{r}^{p^{n-s}} - \DrV{n+r-s}{V^{\perp}} \right) = 0$ and
$\DrV{0}{V} \left( \kat{r}^{p^{n-s}} - \DrV{n+r-s}{V^{\perp}} \right) = 0$
in~$\h*{\pgp}$.
\end{prop}

\begin{proof}
By \Quillen., it suffices to prove these equalities
after restriction to each \maxel.~$M$ of~$G$.  For such an~$M$, let~$I$ be
the corresponding \mti. subspace~$I$ of~$E$.  If
$\dim \Res_I (V) < \dim (V)$, then~$\DrV{0}{V}$, and hence~$\CnV{V}$,
restrict to zero by
Lemma~\ref{lem:Dickson_larger}.
Otherwise $\dim \Res_I (V) = \dim (V)$, and so
Lemma~\ref{lem:Estar_correspondence} tells us that
$\Res_I (V^{\perp}) = I^*$.  Applying Lemma~\ref{lem:Dickson_larger} again,
$\Res_M \left( \DrV{n+r-s}{V^{\perp}} \right) = \DrV{r}{I^*}^{p^{n-s}}$:
but this is also $\Res_M (\kat{r}^{p^{n-s}})$, by Proposition~\ref{prop:res_kz}.
\end{proof}

\begin{coroll}
\label{cor:K_rP_VY}
Let $0 \leq r \leq n-1$; let~$V$ be a $(n + s)$-dimensional subspace of~$E^*$
for some $0 \leq s \leq r$; and let~$Y$ be a $2s$-dimensional subspace
of~$V$.  For any complementary subspace~$W$ of~$Y$ in~$V$,
the equation $\kat{r}^{p^s} P_{V,Y} = \DrV{r+s}{W^{\perp}} P_{V,Y}$
holds in~$\h*G$.  Hence $\kat{r}^{t p^s} \DrV{2s}{V} \in \TY$ for all
$t \geq 1$.
\end{coroll}

\begin{proof}
Observe that $\DrV{0}{W}$~divides $P_{V,Y}$.
The last part follows by Macdonald's Theorem.
\end{proof}

%
%
%

\begin{theorem}
\label{thm:kappa_integral}
For every $0 \leq r \leq n-1$ and every $0 \leq s \leq n$,
\begin{equation}
\displaystyle\prod_{V \in \Z_s} \left( \kat{r}^{p^{n-s}} -
\DrV{n+r-s}{V^{\perp}} \right) = 0 \; .
\end{equation}
In particular,
$\kat{r}^{t p^r} \in \TY$ for all $t \geq 2^{n-r}$.
\end{theorem}

\begin{proof}
We use \Quillen..  Let~$M$ be a \maxel. of~$G$, and $I$~the corresponding
\mtis. of~$E$.
By Proposition~\ref{prop:res_kz}, the restriction of $\kat{r}$~to~$M$
is~$\DrV{r}{I^*}$.
By Proposition~\ref{prop:extendZ_r}, there is
some $V \in \Z_s$ whose restriction to~$I$ has dimension $s = \dim (V)$.
Then by Lemma~\ref{lem:Estar_correspondence},
the restriction of~$V^{\perp}$ is~$I^*$.
Hence $\kat{r}^{p^{n-s}} - \DrV{r}{V^{\perp}}$ restricts to zero by
Lemma~\ref{lem:Dickson_larger}.
The last part follows by Corollary~\ref{cor:K_rP_VY}.
\end{proof}



\begin{thebibliography}{00}

\bibitem{Atiyah}
M.~F. Atiyah.
\newblock Characters and cohomology of finite groups.
\newblock {\em Inst.\ Hautes \'Etudes Sci.\ Publ.\ Math.} \vypp9{1961}{23}{64}.

\bibitem{Benson:poly_invts}
D.~J. Benson.
\newblock {\em Polynomial Invariants of Finite Groups}.
\newblock London Math.\ Soc.\ Lecture Note Ser. no.~\oldstylenums{190}
(Cambridge Univ.\ Press, \oldstylenums{1993}).

\bibitem{Benson-Carlson}
D.~J. Benson and J.~F. Carlson.
\newblock The cohomology of extraspecial groups.
\newblock {\em Bull.\ London Math.\ Soc.} \vypp{24}{1992}{209}{235}.
\newblock Erratum: {\em Bull.\ London Math.\ Soc.} \vyp{25}{1993}{498}.

\bibitem{Green}
D.~J. Green.
\newblock Chern classes and extraspecial groups of order~$p^5$.
\newblock Preprint, \oldstylenums{1995}.

\bibitem{Leary:mod-p}
I.~J. Leary.
\newblock The mod-$p$ cohomology rings of some $p$-groups.
\newblock {\em Math.\ Proc.\ Cambridge Philos.\ Soc.}
\vypp{112}{1992}{63}{75}.

\bibitem{Leary-Yagita}
I.~J. Leary and N.~Yagita.
\newblock Some examples in the integral and Brown--Peterson cohomology of
$p$-groups.
\newblock {\em Bull.\ London Math.\ Soc.} \vypp{24}{1992}{165}{168}.

\bibitem{Lewis}
G.~Lewis.
\newblock The integral cohomology rings of groups of order $p^3$.
\newblock {\em Trans.\ Amer.\ Math.\ Soc.} \vypp{132}{1968}{501}{529}.

\bibitem{Quillen}
D.~Quillen.
\newblock The spectrum of an equivariant cohomology ring:~I.
\newblock {\em Ann.\ of Math.}~\svypp{2}{94}{1971}{549}{572}.

\bibitem{Quillen:mod-2}
D.~Quillen.
\newblock The mod-2 cohomology rings of extra-special 2-groups and the spinor
  groups.
\newblock {\em Math.\ Ann.} \vypp{194}{1971}{197}{212}.

\bibitem{Quillen-Venkov}
D.~Quillen and B.~B. Venkov.
\newblock Cohomology of finite groups and elementary abelian subgroups.
\newblock {\em Topology} \vypp{11}{1972}{317}{318}.

\bibitem{Tezuka-Yagita:varieties}
M.~Tezuka and N.~Yagita.
\newblock The varieties of the mod~$p$ cohomology rings of extra special
  $p$-groups for an odd prime~$p$.
\newblock {\em Math.\ Proc.\ Cambridge Philos.\ Soc.}
\vypp{94}{1983}{449}{459}.


\end{thebibliography}


\end{document}